\documentclass[11pt]{amsart}

\usepackage{amsmath}
\usepackage{amssymb}
\usepackage{graphicx}

\newtheorem{theorem}{Theorem}[section]
\newtheorem{corollary}[theorem]{Corollary}
\newtheorem{lemma}[theorem]{Lemma}

\theoremstyle{definition}
\newtheorem{definition}[theorem]{Definition}
\newtheorem{example}[theorem]{Example}

\newcommand{\argmin}{\operatornamewithlimits{argmin}} 

\newcommand{\RR}{\mathbb{R}}
\newcommand{\Rn}{\mathbb{R}^n}

\newcommand{ \RntoR}{\mathbb{R}^n \rightarrow \mathbb{R}}
\newcommand{ \RtoR}{\mathbb{R} \rightarrow \mathbb{R}}

\newcommand{\Argmin}{\operatorname{Argmin}}

\newcommand{\be}{\begin{enumerate}}
\newcommand{\ee}{\end{enumerate}}
\newcommand{\bi}{\begin{itemize}}
\newcommand{\ei}{\end{itemize}}
\newcommand{\bc}{\begin{center}}
\newcommand{\ec}{\end{center}}
\newcommand{\beq}{\begin{equation}}
\newcommand{\eeq}{\end{equation}}

\newcommand{\bth}{\begin{theorem}}
\newcommand{\bcor}{\begin{corollary}}
\newcommand{\ecor}{\end{corollary}}
\newcommand{\bax}{\begin{axiom}}
\newcommand{\eax}{\end{axiom}}
\newcommand{\bdf}{\begin{definition}}
\newcommand{\edf}{\end{definition}}
\newcommand{\blem}{\begin{lemma}}
\newcommand{\elem}{\end{lemma}}
\newtheorem{exa}{Example}
\newcommand{\bex}{\begin{exa}}
\newcommand{\eex}{\end{exa}}

\def\eps{\varepsilon}
\def\els{\varepsilon_{ls}}

\def\eee {\eps_0,\eps_1,\els}
\newfont{\gothic}{eufm10}

\def\l{\lambda}
\def\lm{\lambda}

\def\cond{\mbox{cond }}
\newcommand{\rsquare}

\def\lr{\lambda_{(r)}}

\def\lr1{\hat\lambda_{(r)}}
\def\lm1{\hat\lambda_{(m-r)}}
\def\bee{\begin{equation}}
\def\eee{\end{equation}}

\def\proof{\noindent {\bf Proof}. \ }
\numberwithin{equation}{section}

\title[Complexity of the  Regularized Newton Method]{Complexity of  the  Regularized Newton Method}

\author[R. Polyak]{Roman A. Polyak}

\address[R. Polyak]{ Department of Mathematics\\
The Technion - Israel Institute of Technology\\
 32000 Haifa, Israel}
\email{{\tt rpolyak@techunix.technion.ac.il\and rpolyak@gmu.edu}}



\begin{document}

\begin{abstract}
 Newton's method for finding unconstrained minimizer of    strictly convex functions, generally speaking, does not converge from any starting point.

We introduce and study the damped regularized Newton's method (DRNM). It converges globally for any strictly convex function, which has a minimizer in $\Rn$. 

Locally DRNM converges with quadratic rate. We characterize the neighborhood of the minimizer, where the quadratic  rate occurs. Based on it we estimate the number of DRNM's steps required for finding an $\varepsilon$- approximation for the minimizer.
\end{abstract}

\maketitle


\section{Introduction}
 Newton's  method, which has been introduced  almost 350 years ago, is still  one of the  basic tools in numerical analysis, variational and control problems, optimization  both constrained and unconstrained, just to mention a few. 

It has  been used not only  as a numerical tool, but also as a  powerful instrument for proving  existence  and uniqueness results.

In particular, Newton-Kantorovich's  method plays a critical role in the classical KAM theory by   Kolmogorov, Arnold and  Mozer (see \cite{1}). Another example is the  proof of   Lusternik's theorem on tangent spaces (see  \cite{3}, \cite{9}).

Newton's method was  the main instrument in the interior point methods (IPMs), which preoccupied the field of optimization for a long time. 

Yu. Nesterov and A. Nemirovski shown that a special  damped Newton's method is particularly efficient for minimization   self - concordant (SC) functions (see \cite{6}, \cite{7}).

They shown that from any starting point  a special  damped Newton's step reduces the SC function value by a constant, which  depends only on the Newton's decrement. The decrement converges to zero.

 By the time it  gets small enough the damped Newton's method practically turns into Newton's method and  generates a sequence, which  converges in  value  with quadratic rate. 

They characterized the size of the minimizer's  neighborhood, where  quadratic rate  occurs. It allows establishing the complexity of the special damped Newton's method for SC function, that is to find the upper bound for the number  of damped Newton's step required for finding an $\varepsilon$- approximation for the minimizer.

For strictly  convex functions, which are not self-concordant, such results, to the best of our knowledge, are unknown. 

The purpose of the paper is to introduce and establish  complexity bounds of  the damped  Newton's  method (DNM) and DRNM   for  minimization of twice continuously differentiable  and strictly convex $f:\Rn\to\RR$.

First, we characterize the Newton's areas for DNM and DRNM. In other words, we estimate the minimizer's neighborhoods, where DNM and DRNM  converges with quadratic rate.

Then we estimate the number of steps needed for DNM's or DRNM's to enter the correspondent Newton's areas.

The key ingredients of our analysis are the   Newton's and  the regularized Newton's  decrements.

On the one hand, the decrements  provide the upper bound  for the distance from the current approximation to the minimizer. Therefore they have been  used in the stopping criteria. 

On the other hand, they provide a lower bound for the   function reduction at each step  at any point, which does  not  belong  to the Newton's or to the regularized Newton's area.

These  bounds were  used to estimate the number of DNM or DRNM steps  needed to get into  the correspondent Newton's areas.

 \section{Newton's Method}
 We start with the classical Newton's method for finding a root of a nonlinear equation
 $$ f(t)=0,$$
 where $f:\RtoR$ has a smooth derivative $ f'$.
 
 Let us  consider $t_{0} \in \RR$   and the linear approximation
 $$\widetilde{f}(t)= f(t_{0})+f'(t_{0})(t-t_{0})=f(t_{0})+f'(t_{0})\Delta t$$
  of $f$ at $  t_{0}$, assuming that $ f'(t_{0}) \neq 0$.
  
 By replacing $f$ with its linear approximation we obtain the following equation  $$ f(t_{0})+f'(t_{0})\Delta t=0$$  for   the Newton's step
 $\Delta t.$\\
 The next approximation is given by formula 
 \begin{equation}
  t=t_{0}+\Delta t = t_{0} - (f'(t_{0}))^{-1} f(t_{0}).\label{Newt2}
 \end{equation}
 By reiterating (\ref{Newt2}) we obtain Newton's method 
 \begin{equation}
 t_{s+1}= t_{s}-(f'(t_{s}))^{-1}f(t_{s})\label{Newt3}
 \end{equation}
  for finding a root of a nonlinear equation  $f(t)=0.$
 
 Let $t^*$ be the root, that is $f(t^*)=0$. Also we assume  $f^{'}(t^*)\neq 0$ and  $f\in C^2$. We consider the expansion of $f$  at $t_s$ with the Lagrange remainder
 \begin{equation}\label{Eq10_6.7.15}
 0=f(t^*)=f(t_s)+f^{'}(t_s)(t^*-t_s)+\frac{1}{2}f^{''}(\hat{t}_s)(t^*-t_s)^2,
 \end{equation}
 where $\hat{t}_s\in[t_s,t^*]$. For $t_s$ close to $t^*$ we have $f^{'}(t_s)\neq 0$, therefore from (\ref{Eq10_6.7.15})  follows 
 $$t^*-t_s+\frac{f(t_s)}{f^{'}(t_s)}=-\frac{1}{2}\frac{f^{''}(\hat{t}_s)}{f^{'}(t_s)}(t^*-t_s)^2.$$
 Using  (\ref{Newt3}) we get 
 \begin{equation}\label{Eq11_6.7.15}
 |t^*-t_{s+1}|=\frac{1}{2}\frac{|f^{''}(\hat{t}_s)|}{|f^{'}(t_s)|}|t^*-t_s|^2.
 \end{equation}
 If $\Delta_s=|t^*-t_s|$ is small, then there exist $a>0$ and $b>0$ independent on $t_s$  that $|f^{''}(\hat{t}_s)|\leq a$ and $|f^{'}(t_s)|>b$.
 Therefore, from (\ref{Eq11_6.7.15}) follows 
 \begin{equation}\label{Eq12_6.7.15}
 \Delta_{s+1}\leq c\Delta^2_s,
 \end{equation}
 where $c=0.5ab^{-1}$.

This is the key characteristic of  Newton's method, which makes the method so important even 350 years after it was originally introduced.
 
 Newton's method has a natural extension for a  nonlinear system of equations
 \begin{equation}
 g(x)=0,\label{Newt4}
 \end{equation}
 where $g:\mathbb{R}^n  \rightarrow \mathbb{R}^n$ is a vector-function with a smooth Jacobian $J(g)=\nabla g: \mathbb{R}^n  \rightarrow \mathbb{R}^n$. The linear approximation of $g$ at $x_{0}$ is given by
 \begin{equation}
 \widetilde{g}(x)=g(x_{0})+\nabla g(x_{0})(x-x_{0}). \label{new5}
 \end{equation}
 We replace $g$ in \eqref{Newt4} by its linear approximation  (\ref{new5}). The Newton's step $ \Delta x$ one  finds by solving the following linear system :
 $$g(x_{0}) +\nabla g(x_{0}) \Delta x=0.$$
 Assuming $\det\nabla g (x_{0}) \neq 0 $ we obtain 
 $$ \Delta x=-(\nabla g(x_0))^{-1}g(x_0).$$
 The new approximation is given by the following formula:
 \begin{equation}
 x=x_0-(\nabla g(x_{0}))^{-1}g(x_{0}).\label{Newt6}
 \end{equation}
 By reiterating (\ref{Newt6}) we obtain   Newton's method 
 \begin{equation}
 x_{s+1}=x_{s}- (\nabla g (x_{s}))^{-1}g(x_{s})\label{Newt7}
 \end{equation}
 for solving a nonlinear system of equations (\ref{Newt4}).
 
 Newton's  method for minimization of $f:\RntoR$ follows directly from (\ref{Newt7}) if 
  instead of unconstrained minimization problem
 \begin{equation}
 \begin{split}
 &\min f(x)\\\label{Newt8}
 &\mbox {s.t. }x\in\Rn 
 \end{split}
 \end{equation}
we consider  the  nonlinear system
 \begin{equation}
 \nabla f(x)=0,\label{Newt9}
 \end{equation}
 which  is the necessary  and sufficient condition  for $x^*$ to be the minimizer in \eqref{Newt8} in case of convex $f$.
 
Vector
 \begin{equation}\label{Eq6_18.9.15}
 n(x)=-(\nabla^2 f(x))^{-1}\nabla f(x)
 \end{equation}
defines  the Newton's direction at $x\in\Rn$.
 
 Application  of  Newton's method (\ref{Newt7}) to the  system (\ref{Newt9}) leads to  the Newton's method 
 \begin{equation}
 x_{s+1}= x_s-(\nabla ^2 f(x_s))^{-1}\nabla f(x_s)=x_s+n(x_s)\label{Newt10}
 \end{equation} 
 for solving (\ref{Newt8}).\\
 Method (\ref{Newt10}) has another interpretation. Let $f: \RntoR$ be twice  differentiable with a positive definite Hessian $\nabla^2 f$.
 
 The quadratic approximation of $f$  at $x_0$ is given by the  formula
 $$\widetilde{f}(x)= f(x_0)+(\nabla f(x_0), x-x_0)+\frac{1}{2} (\nabla ^2 f(x_0)(x-x_0), x-x_0).$$
 Instead  of solving (\ref{Newt8}) let us  find
 $$\bar{x}=\argmin \{ \widetilde{f} (x) \, : x \in \Rn\},$$
 which is equivalent to solving the  following  linear system 
 $$\nabla ^2 f(x_0) \Delta x= - \nabla f(x_0)$$
 for $\Delta x= x-x_0$.\\
We obtain
 $$ \Delta x= n(x_0),$$
 so for the next approximation we have 
 \begin{equation}
 \bar{x}=x_0- (\nabla^2 f(x_0))^{-1} \nabla f(x_0)=x_0+n(x_0) .\label{Newt11}
 \end{equation}
 By reiterating (\ref{Newt11})  we obtain  Newton's   method (\ref{Newt10}) for solving (\ref{Newt8}). 
 
 The local quadratic convergence of both \eqref{Newt7} and \eqref{Newt10} is well known (see \cite{2}, \cite{4}, \cite{7}, \cite{8} and references therein). 
 
 Away from the neighborhood of $x^*$, however, both Newton's methods   \eqref{Newt7} and \eqref{Newt10}  can either oscillate or  diverge.
 \begin{example}
 Consider
 \begin{eqnarray*}
g(t)= \left\{ 
 		\begin{array}{ll}
 		-(t-1)^2+1,\;  & t \geq 0, \\
 		 \;\;\;(t+1)^2-1,\; & t<0.
 		\end{array}
 		 \right.
 \end{eqnarray*}
 \end{example}
 The function $g$ together with $g'$  is continuous on $ (-\infty , \infty ).$   Newton's  method (\ref{Newt3}) converges to the root $t^*=0$ from  any starting point $ t$: $ |t| <\frac{2}{3}$,  oscillates between $t_s=-\frac{2}{3}$ and $t_{s+1}=\frac{2}{3}$, $s=1,2,...$  and  diverges for any  $t$: $ |t| >\frac{2}{3}$.

 \begin{example}\label{ExmNewt2}
 For $ f(t)=\sqrt{1+t^2}$ we have $$f(t^*)= f(0)=\min\{f(t)\, : \, -\infty <t < \infty \}.$$ 
 For the first and second derivative  we have 
 $$f'(t)=t(1+t^2) ^{-\frac{1}{2}}, \,\,\,  f''(t)=(1+t^2)^{-\frac{3}{2}}.$$
 \end{example}
 
 Therefore  Newton's method \eqref{Newt10} is given by the following formula
 \begin{equation}
 t_{s+1}=t_{s}-(1+t_s^2)^{\frac{3}{2}}t_s(1+t_s^2)^{-\frac{1}{2}}= -t_s^3. \label{Newt12}
 \end{equation}
 It follows from (\ref{Newt12}) that   Newton's  method converges from any $t_0 \in (-1,1)$  oscillates between $t_s=-1$ and $t_{s+1}=1$, $s=1,2,...$ and diverges from any $t_0 \notin  [-1,1]$. It also follows from (\ref{Newt12}) that  Newton's method converges from any starting point $t_0 \in (-1,1)$  with the cubic rate, however,  in both examples the convergence area is  negligibly smaller than the area where Newton's method diverges. Note that $f$ is strictly convex in $\RR$ and strongly convex in the neighborhood of $t^*=0$.
 
 Therefore there are three important issues associated with the Newton's  method for unconstrained convex optimization.\\
 First,  to characterize the neighborhood of the solution, where Newton's method converges with  quadratic rate.\\
 Second,   to find such modification of  Newton's method that  generates  convergent sequence  from any starting point and retains  quadratic convergence rate in the neighborhood of the solution.\\
 Third,  to estimate the computational complexity of a globally convergent Newton's and regularized Newton's  methods in terms of the total number of steps required for finding an  $\varepsilon$-approximation for $x^*$.
 
 \section{Local Quadratic Convergence of  Newton's Method}
 We consider a class of convex functions $ f: \RntoR$, that  are strongly convex at $x^*$, that is 
 \begin{equation}
 \label{Eq_4.2.16}
 \nabla^2f(x^*)\succeq mI,
 \end{equation}
  $m>0$ and their Hessian satisfy   Lipschitz condition in the neighborhood of $x^*$. In other words there is $\delta >0$, a ball $B(x^*,\delta)=\{x\in\Rn,\|x-x^*\|\leq\delta\}$ and $M>0$ such that for any $x$ and $y\in B(x^*,\delta)$ we have 
  \begin{equation}
  \label{Eq2_4.2.16}
  \|\nabla^2 f(x)-\nabla^2 f(y)\|\leq M\|x-y\|.
  \end{equation}
  
   The following Theorem characterize the  neighborhood of $x^*$, where  Newton's method converges with quadratic rate. 
   
   There are several ways to proof this fundamental result(see, for example, \cite{2}, \cite{4}, \cite{7}, \cite{8} and references therein). In the following Theorem, which we provide for completeness,  the Newton's area is characterized explicitly through the convexity constant $m>0$ and Lipschitz constant $M>0$ (see \cite{7}). We will use these technique later to characterize the regularized Newton's area.
 \begin{theorem} \label{TNewt1}
 If for $0<m<M$ conditions \eqref{Eq_4.2.16} and \eqref{Eq2_4.2.16} are satisfied,  then for  $\delta=\frac{2m}{3M}$ and any given $x_0 \in B (x^*, \delta)$  the entire sequence $\{x_s\}_{s=0}^{\infty}$ generated by (\ref{Newt10}) belongs $B(x^*, \delta)$ and the following bound holds:
 \begin{equation}
 \|x_{s+1} -x^* \| \leq \frac{M}{2(m-M\|x_s-x^*\|)}\|x_s-x^*\|^2, \, s\geq 1 .\label{Newt15}
 \end{equation}
 
 \end{theorem}
 
 \begin{proof}
 From  (\ref{Newt10}) and $ \nabla  f(x^*)=0$ follows 
 \begin{eqnarray} \label{Newt16}
 x_{s+1}-x^*=x_s-x^*-[\nabla^2 f(x_s)]^{-1} \nabla f(x_s)=\nonumber\\
 =x_s-x^* - (\nabla^2 f(x_s))^{-1} (\nabla f(x_s)-\nabla f(x^*))= \nonumber\\
 =[\nabla ^2 f(x_s)] ^{-1}[\nabla^2 f(x_s)(x_s-x^*)- (\nabla f(x_s)-\nabla f(x^*))].
 \end{eqnarray}
 
 Then we have 
 $$\nabla f(x_s) -\nabla f(x^*)= \int_{0}^{1}\nabla^2 f (x^*+\tau(x_s-x^*))(x_s-x^*)d\tau.$$
 From (\ref{Newt16}) we obtain
 \begin{equation}
 x_{s+1}-x^* =[\nabla^2 f(x_s)]^{-1}H_s (x_s-x^*), \label{Newt17}
 \end{equation}
 where
 $$H_s=\int_0^1[\nabla ^2 f(x_s)-\nabla^2 f(x^* +\tau (x_s-x^*))]d\tau.$$
 Let $\Delta _s =\|x_s -x^*\|$, then using (\ref{Eq2_4.2.16}) we get
 \begin{eqnarray*}
 \|H_s\|=\| \int_0^1[\nabla^2 f(x_s)-\nabla^2 f(x^* +\tau (x_s-x^*))]d\tau \| \\
 \leq \int_0^1\|[\nabla^2 f(x_s)-\nabla^2 f(x^* +\tau (x_s-x^*))\|d\tau \leq \\
 \leq \int_0^1 M\|x_s-x^* -\tau(x_s-x^*)\|d\tau \leq\\
 \leq \int_0^1 M (1-\tau)\|x_s-x^* \|d \tau=\frac{M}{2}\Delta_s.
 \end{eqnarray*}
  Therefore from (\ref{Newt17}) and the latter bound we have

$$ \Delta_{s+1} \leq \|(\nabla^2 f(x_s))^{-1}\| \|H_s\| \|x_s-x^*\| \leq$$
  \begin{eqnarray}
  \frac{M}{2}\|(\nabla ^2 f(x_s))^{-1}\|\Delta^2_s. \label{Newt18}
 \end{eqnarray}
 From (\ref{Eq2_4.2.16}) follows
 $$\|\nabla^2 f(x_s)- \nabla^2f(x^*)\| \leq M\|x_s-x^*\|=M\Delta_s,$$
 therefore   $$\nabla ^2 f(x^*) +M\Delta_sI \succeq \nabla ^2 f(x_s)\succeq \nabla ^2 f(x^*)- M\Delta_sI .$$
From \eqref{Eq_4.2.16} follows
 $$\nabla^2 f(x_s) \succeq \nabla ^2 f(x^*)-M\Delta_sI \succeq (m-M\Delta_s)I.$$
 Hence, for any $ \Delta_s <m M^{-1}$ the matrix $\nabla ^2 f(x_s) $ is positive definite, therefore the  inverse $(\nabla^2 f(x^s))^{-1}$ exists and the following bound holds  
 $$\|(\nabla^2 f(x^s))^{-1}\| \leq \frac{1}{m-M\Delta_s}.$$
 From (\ref{Newt18}) and  the latter bound  follows  
 \begin{equation}
 \Delta_{s+1} \leq \frac{M}{2(m-M\Delta_s)}\Delta^2_s .\label{Newt19}
 \end{equation}
From \eqref{Newt19}   for $\Delta _s < \frac{2m}{3M}$  follows $\Delta_{s+1} < \Delta_{s}$,
 which means that for $\delta = \frac{2m}{3M}$ and any $x_0 \in B(x^*, \delta)$ the entire sequence $\{x_s\}_{s=0}^\infty $ belongs to $ B(x^*, \delta)$ and converges to $x^*$ with the   quadratic rate \eqref{Newt19}.
  
  The proof is completed \qed
  \end{proof}

  The neighborhood $ B(x^*, \delta)$ with $ \delta = \frac{2m}{3M}$ is called   Newton's area.
 
 In the following section we consider a new version of the damped Newton's method, which converges  from any starting point and at the same time  retains  quadratic convergence rate in the Newton's area.
 
 \section{Damped Newton's Method }
 
 To make Newton's method practical we have to guarantee convergence from any starting point. To this end the step length $t>0$ is  attached to the Newton's direction $n(x)$, that is
 \begin{equation}\label{Eq_28.8.15}
 \hat{x}=x+tn(x)=x-t(\nabla^2 f(x))^{-1}\nabla f(x).
 \end{equation}
   The step length $t>0$ has to be adjusted to guarantee a  "substantial  reduction" of $f$  at each $x \notin B(x^*, \delta)$ and $t=1$, when $x \in B(x^*, \delta)$.\\
Method \eqref{Eq_28.8.15} is called the damped Newton's Method (DNM)(see, for example, \cite{2}, \cite{7},  \cite{9})
 
 The following function $\l:\Rn\to \RR_+$:
 \begin{equation}\label{Eq3_4.2.16}
 \l(x)=((\nabla^2f(x))^{-1}\nabla f(x),\nabla f(x))^{0.5}=[-(\nabla f(x),n(x))]^{0.5},
 \end{equation}
  which is called the Newton's  decrement of $f$ at $x\in\Rn$, will play an important role later. 
  
  At this point  we  assume that  $f:\Rn\to\RR$ is strongly convex  and its  Hessian $\nabla^2f$ is Lipschitz continuous, that is, there exist $\infty>M>m>0$ that 
   \begin{equation}
   \label{Eq_5.2.16}
   \nabla^2f(x)\succeq mI
   \end{equation} 
   and 
   \begin{equation}
   \label{Eq2_5.2.16}
   \|\nabla^2 f(x)-\nabla^2 f(y)\|\leq M\|x-y\|
   \end{equation} 
   are satisfied for any $x$ and $y$ from $\Rn$.
   
  Let $x_0\in\Rn$ be a starting point. 
  
      Due to \eqref{Eq_5.2.16} the sublevel set $\mathcal{L}_0=\{x\in\Rn:f(x)\leq f(x_0)\}$  is bounded for any given $x_0\in \Rn$. Therefore     from \eqref{Eq2_5.2.16} follows existence $L>0$ that 
   \begin{equation}
     \label{Eq3_5.2.16}
     \|\nabla^2 f(x)\|\leq L
     \end{equation} 
     is taking place.
   
  We also assume that $\varepsilon>0$ is small enough, in particular, 
  \begin{equation}\label{Eq_19.2.16}
  0<\varepsilon<  m^2L^{-1}
  \end{equation}
  holds.
 
 We are ready to describe our version of DNM.
 
 Let $x_0 \in \Rn$ be a starting point and $0<\varepsilon<\delta $ be the required accuracy.  Set $x:=x_0$
 \begin{enumerate}
 \item[1.] find Newton's direction $n(x)$;
 \item[2.] if the following inequality
\begin{equation}
f(x+n(x))\leq f(x)+0.5(\nabla f(x),n(x))
\end{equation}
 holds, then set $t(x):=1$, otherwise  set $t(x)=m(2L)^{-1}$;
 \item[3.] set $x:=x+t(x)n(x)$;
 \item[4.] if $\l(x)\leq\varepsilon^{1.5}$, then $x^*:=x$,  otherwise go 1.
 \end{enumerate}
 The following Theorem proves  global convergence of the DNM  1.-4. and establishes the upper bound for the total number of DNM steps require for finding $\varepsilon$-approximation for $x^*$.

 \section{Global Convergence of the DNM and its Complexity}

 \begin{theorem}\label{TNewt2}
 If $f:\RntoR$ is twice  differentiable and conditions (\ref{Eq_5.2.16}) and (\ref{Eq2_5.2.16}) are satisfied, then for $\delta=\frac{2}{3}\frac{m}{M}$ it takes
 \begin{equation}\label{Eq7_18.9.15}
 N_0=9\frac{L^2M^2}{m^5} (f(x_0)-f(x^*)).
 \end{equation}
 DNM steps to find $x\in B(x^*,\delta)$ by using DNM.
 \end{theorem}
 \proof
 From \eqref{Eq3_5.2.16} follows
 \begin{equation}\label{Eq_2.3.16}
 \nabla^2 f(x)\preceq LI.
 \end{equation} 
  On other hand, from (\ref{Eq_5.2.16}) follows the existence of the inverse $(\nabla ^2 f(x))^{-1}$. Therefore from (\ref{Eq_2.3.16}) follows
  \begin{equation}\label{Eq2_2.3.16}
  (\nabla^2 f(x))^{-1}\succeq L^{-1}I.
  \end{equation}
  From (\ref{Eq3_4.2.16}) and (\ref{Eq2_2.3.16}) we obtain the following lower bound for the Newton's decrement
  \begin{equation}\label{Eq3_2.3.16}
  \lambda (x)=(\nabla^2 f(x)^{-1}\nabla f(x), \nabla f(x))^{0.5}\geq
  \end{equation}
  $$\geq (L^{-1}\|\nabla f(x)\|^2)^{0.5}=L^{-0.5}\|\nabla f(x)\|.$$
  From (\ref{Eq_5.2.16}) we have
  $$\|\nabla f(x)\|\|x-x^*\|\geq (\nabla f(x)-\nabla f(x^*),x-x^*)\geq m\|x-x^*\|^2$$
  or
  \begin{equation}\label{Eq4_2.3.16}
  \|\nabla f(x)\|\geq m\|x-x^*\|.
   \end{equation}
    From (\ref{Eq3_2.3.16}) and (\ref{Eq4_2.3.16}) we obtain
    \begin{equation}\label{Eq5_2.3.16}
    \lambda (x)\geq L^{-0.5}m\|x-x^*\|.
    \end{equation}
    From \eqref{Eq_19.2.16} and  the stopping criteria 4. follows 
    $$(m^2 L^{-1})^{0.5}\varepsilon\geq \varepsilon ^{1.5}\geq \lambda (x)\geq m L^{-0.5}\|x-x^*\|,$$
    or $\|x-x^*\|\leq \varepsilon$, which  justifies the stopping criteria 4.
    
 On the other hand,  Newton's decrement defines the lower bound for the function reduction at each step.
 
 In fact, for  Newton's directional derivative from \eqref{Eq6_18.9.15},  \eqref{Eq3_4.2.16} and \eqref{Eq_5.2.16} follows
$$ \varphi'(0)= \frac{df(x+tn(x))}{dt}|_{t=0}=(\nabla f(x), n(x))=$$
 \begin{equation}
-(\nabla ^2 f(x)n(x),n(x)) \leq -m\|n(x)\|^2 .\label{Newt30}
 \end{equation}
 
 Due to the strong convexity of $\varphi (t)=f(x+tn(x))$  the derivative   $\varphi ^{'}(t)=(\nabla f(x+tn(x)),n(x))$ is monotone  increasing in $t>0$, so there is $t(x)>0$ such that 
 \begin{equation}\label{Newt30b}
 0>(\nabla f(x+t(x)n(x)), n(x))\geq\frac{1}{2}(\nabla f(x),n(x)),
 \end{equation}
 otherwise $(\nabla f(x+tn(x)),n(x))<\frac{1}{2}(\nabla f(x),n(x))\leq -\frac{1}{2} m\|n(x)\|^2,t>0$ and  $\inf f(x)=-\infty$, which is impossible for a strongly convex function $f$.
 
 It follows from (\ref{Newt30}), (\ref{Newt30b}) and  monotonicity of $\varphi'(t)$ that for any  $t\in[0,t(x)]$ we have 
 $$\frac{df(x+tn(x))}{dt}=(\nabla f(x+tn(x)),n(x))\leq\frac{1}{2}(\nabla f(x), n(x)).$$
 Therefore 
 $$f(x+t(x)n(x)) \leq f(x) +\frac{1}{2} t(x)( \nabla f(x), n(x)). $$
Keeping in mind \eqref{Eq3_4.2.16} we obtain
 \begin{equation}
 f(x) -f( x+t(x)n(x))\geq \frac{1}{2} t(x)\lambda ^2 (x).  \label{Newt31}
 \end{equation}
 Combining  (\ref{Newt30}) and (\ref{Newt30b}) we obtain
 $$(\nabla f(x+t(x)n(x))-\nabla f(x), n(x)) \geq \frac{m}{2} \|n(x)\|^2.$$ 
 Therefore, there is $0<\theta (x) <1$ such that 
 $$t(x)(\nabla^2 f(x+\theta (x) t(x)n(x))n(x),n(x))=t(x)(\nabla^2f(\cdot)n(x),n(x))\geq \frac{m}{2} \|n(x) \|^2,$$
 or
 $$t(x) \|\nabla ^2 f(\cdot) \| \|n(x)\|^2 \geq \frac{m}{2} \|n(x)\|^2.$$
From \eqref{Eq_5.2.16} follows 
 \begin{equation}
 t(x) \geq \frac{m}{2L},  \label{Newt33}
 \end{equation}
 which justifies the choice of step length $t(x)$ in the DNM 1.-4..
 
 Hence, from (\ref{Newt31}) and \eqref{Newt33}   we obtain the following lower bound for the function reduction per step 
 \begin{equation}
 \Delta f(x) =f(x) -f(x +t(x)n(x))\geq \frac{m}{4L} \lambda ^2 (x) ,\label{Newt34}
 \end{equation}
 which together with the lower bound \eqref{Eq5_2.3.16} for the Newton's decrement $\lambda(x)$ leads to 
 \begin{equation}
 \Delta f(x) = f(x)-f(x+t(x)n(x) )\geq \frac{m^3}{4L^2} \|x-x^*\|^2 .\label{Newt35}
 \end{equation}
 It means that for any $x\notin B(x^*,\delta)$ the function reduction at each step is proportional to the square of the distance between current approximation $x$ and the  solution $x^*$. 
 
 In other words, "far from" the solution   Newton's step produces a  "substantial" reduction of the function  value similar to one of  the gradient method.
 
 For  $x\notin B(x^*, \delta)$ we have $\|x-x^*\| \geq \frac{2m}{3M}$, therefore from (\ref{Newt35}) we obtain
 $\Delta f(x) \geq \frac{1}{9} \frac{m^5}{L^2M^2}$. So it takes at most $$ N_0=9\frac{L^2M^2}{m^5}(f(x_0)-f(x^*))$$ Newton's steps to obtain  $x\in B(x^*,\delta)$ from a given starting point $ x_0\in \Rn$. 
  The proof is  completed. \qed
  
 From Theorem \ref{TNewt1} follows  that $O(\ln \ln \varepsilon ^{-1} ) $  steps needed to find an $\varepsilon$- approximation to $x^*$ from any $x\in B(x^*,\delta)$, where $0<\varepsilon< \delta$ is the required accuracy.
  Therefore the total number of Newton's steps required for finding an  $\varepsilon$-approximation  to the optimal solution $x^*$ from a  starting point $x_0 \in \Rn$ is 
  $$N=N_0+O(\ln\ln\varepsilon ^{-1}).$$

  The bound \eqref{Eq7_18.9.15} is similar to (9.40) from \cite{2}, but the proof is   based on our version of DNM and  the explicit characterization of the Newton's  area. It  allows to extend the proof for the regularized Newton's method \cite{10}. 
 
 The DNM  requires  an a priori knowledge of two parameters $m$ and $L$ or their corresponding lower and upper bounds.
 
 The following version of DNM is free from this  requirement. 
 To adjust  the  step length $t>0$ we use the backtracking line search.
 
 The  inequality 
 \begin{equation}
  f(x+tn(x)) \leq f(x) +\alpha t(\nabla f(x), n(x)) \label{Newt37}
 \end{equation}
 with $0<\alpha \leq 0.5$ 
 is called the Armijo condition.
 
 Let $0< \rho <1$, the backtracking line search consist of the following steps.
 \begin{enumerate}
 \item[1.] For $t >0$ check (\ref{Newt37}). If (\ref{Newt37}) holds go to 2. If not set $t:=t\rho$ and repeat it until  (\ref{Newt37}) holds, then go to 2. 
 \item[2.] set  $t(x):=t$, $x:=x+t(x)n(x)$
 \end{enumerate}
 We are ready to describe another version of DNM, which does not requires an a priori knowledge of the parameters $m$ and $L$ or their lower and upper bounds.
 
 Let $x_0 \in \Rn$ be a starting point and $0<\varepsilon<<\delta $ be the required accuracy. 
 \begin{enumerate}
 \item[1.] Compute Newton's direction
 \begin{equation}n(x) =-(\nabla ^2 f(x)) ^{-1} \nabla f(x);\end{equation}
 \item[2.] set $t:=1$, use the backtracking line search until
 $$f(x+tn(x)) \leq f(x) +0.5 t(\nabla f(x), n(x));$$
 \item[3.] set $t(x):=t$, $x:=x+t(x)n(x)$;
 \item[4.] if $\lambda (x) \leq \varepsilon^{1.5}$ then $x^*:= x$  otherwise 
 go 1.
 \end{enumerate}
 
The complexity of the DNM with backtracking line search can  be established using arguments similar to those in Theorem \ref{TNewt2}

 Unfortunately, in the absence of  strong convexity of $f:\Rn\to\RR$   Newton's method might not converge from any starting point. 
 
 In case of Example \ref{ExmNewt2}  Newton's method does not converge from any  $t\notin(-1,1)$ in spite of $f(t)=\sqrt{1+t^2}$ being strongly  convex  and smooth enough in the neighborhood of $t^*=0$.
 
 In the following section we consider the Regularized Newton's Method (RNM)(see \cite{10}), which eliminates the basic  drawback of the Classical Newton's Method. It  generates a  converging sequence  from any starting point $x_0 \in \Rn$ and retains   quadratic convergence rate in  the regularized Newton's area, which we will characterize  later.

 \section{Regularized Newton's Methods}

 Let $f\in C^2$ be a convex function in $\Rn$.
 We assume that the optimal set $X^*=\Argmin\{f(x):x\in \Rn\}$ is not empty and bounded.
 
 The corresponding regularized  at the point $x \in \Rn$ 
 function $F_x:\RntoR$ is defined by the following formula
 \begin{equation}\label{1.2}
 F_x(y) = f(y) + \frac{1}{2} \parallel \nabla f(x) \parallel
 \parallel y - x \parallel^2.
 \end{equation}
 
 For any $x \notin X^*$ we have $||\nabla f(x)|| > 0$, therefore for any convex function $f:\RntoR$ the regularized function $F_x$ is strongly convex in $y$ for any  $x\notin X^*$. If  $f$ is strongly convex at $x^*$, then  the regularized function $F_x$ is strongly convex in $\Rn$.
 The following properties  of  $F_x$ are direct
 consequences of the definition \eqref{1.2}.
 
 \noindent
 $1^\circ. \hspace{0.3cm} F_x(y)_{|y=x} =  f(x),$ \\
 $2^\circ. \hspace{0.3cm} \nabla_y F_x(y)_{|y=x} = \nabla f(x),$\\
 $3^\circ. \hspace{0.3cm} \nabla_{yy}^2 F_x(y)_{|y=x} = \nabla^2
 f(x) + ||\nabla f(x)|| I = H(x)$, 
 
 where $I$ is the identical  matrix
 in $\Rn$.
 
 For any $x \notin X^*$, the inverse $H^{-1}(x)$ exists for any convex $f\in C^2$. Therefore the
 regularized Newton's step 
 \begin{equation} \label{1.3}
 \hat{x} = x - (H(x))^{-1} \nabla f(x)
 \end{equation}
 can be performed for any convex $f\in C^2$ from any starting point $x \notin X^*$.
 
 We start by showing that the regularization (\ref{1.2}) improves the "quality" of the Newton's direction as well 
 the condition number of the Hessian $\nabla^2 f(x)$ at any $ x \in \Rn$ that $x \notin X^*$.
 
  We assume  at this point that for any
 given $x \in \Rn$ there exist $0 \leq  m(x) < M(x) < \infty$
 such that
 \begin{equation} \label{1.4}
 m(x) ||y||^2 \le (\nabla^2 f(x) y, y) \le M(x) ||y||^2
 \end{equation}
 holds  for any $ y \in \Rn$.
 
 The condition number of the Hessian $\nabla^2 f$ at $x\in\Rn$ is 
 $$\cond \nabla^2 f(x)=m(x)(M(x))^{-1}.$$
 Along with the regularized Newton's step (\ref{1.3}), we consider
 the classical Newton's step 
 \begin{equation} \label{1.5}
 \hat{x} = x - (\nabla^2 f(x))^{-1}\nabla f(x).
 \end{equation}
The regularized Newton's direction (RND) 
 $r(x)$ is defined by the   system
 \begin{equation}\label{1.6}
 H(x) r(x) = -\nabla f(x).
 \end{equation}
 The "quality" of any direction $d$ at $x\in\Rn$ is define by the following number
 $$0\leq q(d)=-\frac{(\nabla f(x),d)}{\|\nabla f(x)\|\cdot \|d\|}\leq 1.$$
 For the steepest descent direction $d(x) =
  -\nabla f(x) \parallel 
  \nabla f(x)\parallel^{-1} $ we have the best local descent direction and $q(d(x)) = 1$.
 The ``quality'' of the classical  Newton's direction is  defined by the
 following number
 \begin{eqnarray}
q(n(x)) = -\frac{(\nabla f(x), n(x))}{||\nabla f(x)||\cdot ||n(x)||} .
 \end{eqnarray}

 For the RND $r(x)$ we have
 \begin{eqnarray}\label{Eq21.6.16}
 q(r(x)) = -\frac{(\nabla f(x), r(x))}{||\nabla f(x)|| \cdot||r(x)||}.
 \end{eqnarray}

 The following theorem establishes the lower bounds for $q(r(x))$
 and $q(n(x))$. It shows that the regularization (\ref{1.2})
 improves the condition number of the Hessian $\nabla^2 f$ for
 all $x \in \Rn, x \notin X^*$ (see \cite{10}).
 
 \begin{theorem}
 Let $f :\RntoR$ be a twice continuous differentiable  convex function and    the bounds   (\ref{1.4}) hold,
 then:
 \begin{enumerate}
 \item[1.]
 \begin{eqnarray*}
 1 \ge q(r(x)) \ge (m(x) + ||\nabla f(x)||) (M(x) + ||\nabla
 f(x)||)^{-1}\\ = \cond\hspace{0.1cm} H(x) > 0 
 \hspace{1cm} \mbox{ for any } x \not\in X^*.
 \end{eqnarray*}
  \item[2.]
  \begin{eqnarray*}
 1 \ge q(n(x)) \ge m(x) (M(x))^{-1} = \cond \hspace{0.1cm}
 \nabla^2 f(x)\\
  \hspace{0.2cm} \mbox{ for any } x \in \mathbb{R}^n.
 \end{eqnarray*}
 \item[3.] $$\cond
  \hspace{0.1cm} H(x)-\cond \hspace{0.1cm} \nabla^2 f(x)  =$$
 \begin{eqnarray}\label{Eq8_18.9.15}
 \begin{split}
  ||\nabla f(x)|| (1-\cond \hspace{0.1cm}
 \nabla^2 f(x)) (M(x) + ||\nabla f(x)||)^{-1} > 0 
  \end{split}
 \end{eqnarray}
for any  $ x \not\in X^*$, $\cond \hspace{0.1cm} \nabla^2 f(x)
  < 1$.
 \end{enumerate}
 \end{theorem}
 
 \proof
 
 \begin{enumerate}
 \item[1.] From (\ref{1.6}), we obtain
 \begin{eqnarray} \label{1.9}
 ||\nabla f(x)|| \le ||H(x)||\cdot ||r(x)||.
 \end{eqnarray}
 Using the right inequality (\ref{1.4}) and $3^\circ$, we have
 \begin{eqnarray} \label{1.10}
 ||H(x)|| \le M(x) + ||\nabla f(x)||,
 \end{eqnarray}
 From (\ref{1.9}) and (\ref{1.10}) we obtain
 \begin{eqnarray*}
 ||\nabla f(x)|| \le (M(x) + ||\nabla f(x)||) ||r(x)||.
 \end{eqnarray*}
 
From \eqref{1.6} the left inequality (\ref{1.4}) and $3^{\circ}$ follows 
$$-(\nabla f(x), r(x))=(H(x)r(x),r(x))\geq (m(x)+\|\nabla f(x)\|)\|r(x)\|^2.$$
Therefore from \eqref{Eq21.6.16} follows
 \begin{eqnarray*}
 q(r(x)) \ge (m(x) + ||\nabla f(x) ||)(M(x) + ||\nabla f(x)||)^{-1}
 = \mbox{cond} \hspace{0.1cm} H(x).
 \end{eqnarray*}
 \item[2.] Now let us  consider the Newton's direction $n(x)$. From
 (\ref{1.5}), we have
 \begin{eqnarray} \label{1.11}
 \nabla f(x) = - \nabla^2 f(x) n(x),
 \end{eqnarray}
 therefore,
 \begin{eqnarray*}
 -(\nabla f(x), n(x)) = (\nabla^2 f(x) n(x), n(x)).
 \end{eqnarray*}
From \eqref{1.11} left inequality of (\ref{1.4}), we obtain 
 \begin{equation}\label{1.12}
 q(n(x)) = -\frac{(\nabla f(x), n(x))}{||\nabla f(x)|| \cdot||n(x)||}
 \ge m(x) ||n(x)||\cdot ||\nabla f(x)||^{-1}.
 \end{equation}
 From (\ref{1.11}) and the right  inequality in (\ref{1.4}) follows
 \begin{equation}\label{1.13}
 || \nabla f(x)|| \le ||\nabla^2 f(x)|| \cdot||n(x)|| \le M(x) ||n(x)||.
 \end{equation}
 Combining (\ref{1.12}) and (\ref{1.13}) we have
 \begin{eqnarray*}
 q(n(x)) \ge \frac{m(x)}{M(x)} = \mbox{cond} \hspace{0.1cm}
 \nabla^2 f(x).
 \end{eqnarray*}
 \item[3.] Using the formulas for the condition numbers of $\nabla^2
 f(x)$ and $H(x)$ we obtain (\ref{1.4})
 \qed
 \end{enumerate}
 
  \begin{corollary}
  
 The regularized Newton's direction $r(x)$ is a
  decent direction for any convex $f: \Rn\to\RR$, whereas the  classical Newton's direction $n(x)$ exists and it is a
  decent direction only if $f$ is a strongly convex at $x\in \Rn$. 
  \end{corollary}
  
 Under condition \eqref{Eq_4.2.16}  and \eqref{Eq2_4.2.16}  the RNM retains the local quadratic  convergence rate, which is
  typical for the Classical Newton's method.
 
On the other hand, the regularization (\ref{1.2})
 allows to establish global convergence  and estimate complexity of the 
  RNM, when the original function is only strongly convex at $x^*$.

 \section{Local Quadratic Convergence Rate of the RNM}
  In this section we consider the RNM  and determine the   neighborhood of the minimizer, where the RNM converges with  quadratic rate.

 Along with  assumptions (\ref{Eq_4.2.16}) and   (\ref{Eq2_4.2.16}) for the Hessian $\nabla^2f$  we will use the Lipschitz condition  for the gradient $\nabla f$ 
 \begin{equation}
 \label{Eq4_5.2.16}
 \|\nabla f(x)-\nabla f(y)\|\leq L\|x-y\|,
 \end{equation}
 which is equivalent to  (\ref{Eq3_5.2.16}).
 
 The RNM  generates a sequence $\{x_s\}_{s=0}^{\infty}$: 
 \begin{equation}
 x_{s+1}=x_s-\left[\nabla^2 f(x_s)+\|\nabla f(x_s)\| I\right] ^{-1} \nabla f(x_s).\label{1.14}
 \end{equation}
 
 The following Theorem  characterizes the regularized Newton's area.
 \begin{theorem}\label{NTheorem}
 If  (\ref{Eq_4.2.16}),  (\ref{Eq2_4.2.16}) and \eqref{Eq4_5.2.16}  hold, then for $ \delta=\frac{2}{3} \frac{m}{M+2L}$ and any $x_0 \in B(x^*, \delta)$ as a starting point, the sequence $\{ x_s\}_{s=0}^{\infty} $ generated by RNM  (\ref{1.14}) belongs to $B(x^*, \delta)$ and the following bound holds:
 \begin{equation}
 \Delta_{s+1}=\|x_{s+1}-x^*\| \leq \frac{M+2L}{2}\cdot \frac{1}{m-(M+2L)\Delta_s}\Delta_s^2,\, s\geq 1.\label{1.14a}
 \end{equation}
 \end{theorem}
 \proof From (\ref{1.14}) follows
 
 \begin{eqnarray*}
 x_{s+1}-x^* =x_s-x^*-\left[ \nabla ^2 f(x_s)+\|\nabla f(x_s)\|I\right] ^{-1}(\nabla f(x_s)- \nabla f(x^*)).
 \end{eqnarray*}
 Using 
 $$\nabla f(x_s)- \nabla f(x^*) = \int_{0}^1 \nabla ^2 f(x^*+\tau(x_s-x^*))(x_s-x^*)d\tau, $$
 we obtain 
 \begin{equation}
 x_{s+1}-x^*= \left[ \nabla ^2 f(x_s) +\|\nabla f(x_s)\|I\right] ^{-1} H_s (x_s-x^*), \label{1.15}
 \end{equation}
 where
 $$H_s =\int_0^1 (\nabla^2 f(x_s) +\|\nabla f(x_s) \| I -\nabla ^2 f(x^* +\tau (x_s -x^*)))d\tau.$$
 From  (\ref{Eq2_4.2.16}) and (\ref{Eq4_5.2.16}) follows  
 \begin{eqnarray*}
 \|H_s\|=\|\int_0^1 \left( \nabla ^2 f(x_s) +\|\nabla f(x_s) \|I- \nabla ^2 f(x^*+\tau(x_s-x^*))\right)d\tau\|\\
 \leq \|\int_0^1 (\nabla ^2f(x_s)-\nabla^2 f(x^*+ \tau (x_s-x^*)))d\tau \|+ \int_0^1 \|\nabla f(x_s)\| d\tau\\
 \leq \int_0^1 \| \nabla^2 f(x_s)-\nabla ^2 f(x^*+\tau(x_s-x^*))\|d\tau +\int_0 ^1 \|\nabla f(x_s)-\nabla f(x^*) \| d\tau\\
 \leq \int_0^1 M \|x_s-x^* -\tau(x_s-x^*)\|d\tau+\int_0^1 L\|x_s-x^*\|d\tau
 \end{eqnarray*}
 \begin{equation}\label{Eq5_5.2.16}
 =\int_0^1 (M(1-\tau)+L) \|x_s-x^* \|d\tau=\frac{M+2L}{2}\|x_s-x^*\|.
  \end{equation}
  
 From (\ref{1.15}) and \eqref{Eq5_5.2.16} we have
 \begin{eqnarray}
 \Delta_{s+1}=\|x_{s+1}-x^*\| \leq \|\left( \nabla ^2 f(x_s)+\|\nabla f(x_s)\|I\right) ^{-1} \|\cdot \|H_s\|\cdot \|x_s-x^*\| \nonumber\\
 \leq\frac{M+2L}{2}\|(\nabla ^2 f(x_s)+\|\nabla f(x_s)\| I)^{-1}\|\Delta_s^2 .\label{1.16}
 \end{eqnarray}
 From (\ref{Eq2_4.2.16}) follows 
 \begin{equation}
 \|\nabla^2 f(x_s)-\nabla^2 f(x^*) \| \leq M\| x_s-x^* \|=M\Delta_s,\label{1.17}
 \end{equation}
 therefore we have  
 \begin{equation}\label{Eq_29.8.15}
 \nabla ^2 f(x^*) +M\Delta_s I\succeq \nabla ^2 f(x_s) \succeq \nabla^2 f(x^*)-M\Delta_sI.
 \end{equation}
 From  (\ref{Eq_4.2.16}) and \eqref{Eq_29.8.15} we obtain
 $$\nabla ^2 f(x_s) +\|\nabla f(x_s) \| I \succeq (m+\|\nabla f(x_s)\| -M \Delta _s)I.$$
 
 Therefore for  $\Delta _s < \frac{m+\|\Delta f(x_s)\|}{M} $ the matrix $\nabla ^2 f(x_s) + \|\nabla f(x_s)\| I$ is positive definite, therefore its inverse exists and we have 
 \begin{eqnarray}
 \| (\nabla^2 f(x_s) +\|\nabla f(x_s) \|I ) ^{-1} \| \leq \frac{1}{m+\|\nabla f(x_s)\|-M\Delta _s}\leq \nonumber\\
 \frac{1}{m-M\Delta _s}.\label {1.18}
 \end{eqnarray} 
 For $\Delta_s\leq\frac{2}{3}\frac{m}{M+2L}$ from (\ref{1.16}) and (\ref{1.18}) follows 
 \begin{equation}\label{Eq_31.5.16}\Delta_{s+1} \leq \frac{M+2L}{2}\frac{1}{m-(M+2L)\Delta_s}\Delta_s^2.
 \end{equation}
 Therefore from \eqref{Eq_31.5.16} for $0< \Delta_s \leq \frac{2}{3}\frac{m}{M+2L}<\frac{m+\|\nabla f(x_s)\|}{M} $ we obtain 
 $$\Delta_{s+1} \le \frac{3(M+2L)}{2}\frac{1}{m} \Delta_s^2 \leq \Delta_s.$$
 Hence, for   $\delta = \frac{2}{3} \frac{m}{M+2L}$ and  any $x_0 \in B(x^*, \delta)$ as a starting point the sequence $\{x_s\}^{\infty}_{s=0}$ generated by \eqref{1.14} belongs to $B(x^*,\delta)$ and the bound \eqref{1.14a} holds.
 
  The proof of Theorem \ref{NTheorem} is completed. \qed 
  \begin{corollary}\label{C_31.5.15}
  Under conditions of Theorem \ref{NTheorem} for  $\delta=\frac{2}{3}\frac{m}{M+2L}$ and any $x\in B(x^*,\delta) $   the Hessian $\nabla^2 f(x)$ is positive definite and 
  \begin{equation}\label{Eq5_12.3.16}
  \nabla^2 f(x)\succeq m_0I,
  \end{equation}
  where $m_0=m(\frac{1}{3}M+2L)(M+2L)^{-1}$
  \end{corollary}
  In fact, from \eqref{Eq2_5.2.16} follows 
  $$\nabla ^2 f(x^*)+M\Delta xI\succeq \nabla ^2 f(x)\succeq \nabla ^2 f(x^*)-M\Delta xI,$$ so for any $x\in B(x^*,\delta)$ we have 
  $$\nabla ^2 f(x)\succeq \left(m-\frac{2}{3}\frac{Mm}{M+2L}\right)I=m(\frac{1}{3}M+2L)(M+2L)^{-1}I=m_0I.$$
  From the letter inequality follows 
  $$\|\nabla f(x)\|\|x-x^*\|\geq(\nabla f(x)-\nabla f(x^*),x-x^*)\geq m_0\|x-x^*\|^2,$$
 that is for any $x\in B(x^*,\delta)$ we have 
  \begin{equation}\label{Eq3_12.3.16}
  \|\nabla f(x)\|\geq m_0\|x-x^*\|.
  \end{equation}
 
 It follows from Theorem \ref{NTheorem} that $B(x^*, \delta)$ with $\delta=\frac{2}{3} \frac{m}{M+2L}$ is the Newton's area for the  RNM.
 
 So it takes $O(\ln \ln \varepsilon ^{-1} )$ regularized Newton's  steps to find an $\varepsilon$-approximation for $x^*$ from any  $x \in B(x^*, \delta)$ as a starting point.

 To make  the RNM globally convergent we have to replace the RNM by DRNM and  adjust the step length. It can be done by backtracking line search, using Armijo condition (\ref{Newt37}) with Newton's direction $n(x)$ replaced by regularized Newton's direction $r(x)$. In the following section we introduce another version of the DRNM and   estimate the number of RNM steps required for finding $x\in \mathbb{B}(x^*,\delta)$  from any  given starting point $x_0 \in \Rn$. 

  \section{Damped Regularized Newton's Method}
 Let us  consider the regularized Newton's decrement 
\begin{equation}
\label{Eq2_12.3.16}
\lambda_r(x)=(H^{-1}(x)\nabla f(x), \nabla f(x))^{\frac{1}{2}} =[-(\nabla f(x), r(x))]^{\frac{1}{2}}.
\end{equation}
 
 We  assume that $\varepsilon>0$ is small enough, in particular, 
 \begin{equation}\label{Eq4_12.3.16}
 0<\varepsilon^{0.5}<  m_0(L+\|\nabla f(x)\|)^{-0.5},
 \end{equation}
 for $\forall x\in \mathcal{L}_0.$
 
 From \eqref{Eq3_5.2.16} follows
 \begin{equation}\label{Eq_12.3.16}
 (\nabla ^2 f(x)+\|\nabla f(x)\|I)\preceq (L+\|\nabla f(x)\|)I.
 \end{equation}
 On the other hand, for any $x\in B(x^*,\delta)$ from the Corollary \ref{C_31.5.15} we have 
 $$\nabla ^2 f(x)+\|\nabla f(x)\|I\succeq (m_0+\|\nabla f(x)\|)I.$$
 Therefore the inverse $(\nabla ^2 f(x)+\|\nabla f(x)\|I)^{-1}$ exists and from \eqref{Eq_12.3.16} we obtain 
 $$H^{-1}(x)=(\nabla ^2 f(x)+\|\nabla f(x)\|I)^{-1}\succeq(L+\|\nabla f(x)\|)^{-1}I.$$
 Therefore  from \eqref{Eq2_12.3.16}  for any $x\in B(x^*,\delta)$ we have
 $$\lambda_{(r)}(x)=(H^{-1}(x)\nabla f(x),\nabla f(x))^{0.5}\geq(L+\|\nabla f(x)\|)^{-0.5}\|\nabla f(x)\|,$$
 which together with \eqref{Eq3_12.3.16}  leads to 
 $$\l_{(r)}(x)\geq m_0(L+\|\nabla f(x)\|)^{-0.5}\|x-x^*\|.$$
 Then from $\l_{(r)}(x)\leq \varepsilon^{1.5}$ and \eqref{Eq4_12.3.16} follows 
 $$m_0(L+\|\nabla f(x)\|)^{-0.5}\varepsilon\geq\varepsilon^{1.5}\geq \l_{(r)}(x)\geq m_0(L+\|\nabla f(x)\|)^{-0.5}\|x-x^*\|$$
 or
 $$\|x-x^*\|\leq\varepsilon, \forall x\in B(x^*,\delta).$$
 Therefore $\l_{(r)}(x)\leq \varepsilon^{1.5}$ can be used as a stopping criteria.
 
 We are ready to  describe the DRNM.
 
 Let $x_0 \in \Rn$ be a starting point and   $0<\varepsilon< \delta$ be the required accuracy, set $x:=x_0$.
 \begin{enumerate}
 \item[1.] Compute the regularized Newton's direction $r(x)$ by solving the system (\ref{1.6});
 
 \item[2.] if the following inequality 
\begin{equation}
f(x+tr(x))\leq f(x) +0.5 (\nabla f(x), r(x))
\end{equation}
holds, then  set $t(x):=1$, otherwise set $t(x):=(2L)^{-1}\|\nabla f(x)\|$; 
 \item[3.] $x:=x+t(x)r(x)$;
 \item[4.]   if $\lambda_r(x)\leq \varepsilon^{1.5}$, then $x^*:=x$,
  otherwise  go to 1.
 \end{enumerate}
 The global convergence and the complexity of the DRNM we consider in the following section.


 \section{Complexity of the DRNM}
 We assume  that conditions  (\ref{Eq_4.2.16}) and  (\ref{Eq2_4.2.16}) are satisfied.
 Due to (\ref{Eq_4.2.16}) the solution $x^*$ is unique. Hence, from  convexity $f$  follows  that for any given starting point $x_0\in\Rn$  the  sublevel set $\mathcal{L}_0$ is bounded, therefore   there is $L>0$ such that (\ref{Eq3_5.2.16})  holds on $\mathcal{L}_0$.
 
 Let $\mathbb{B}(x^*,r)=\{x\in\Rn:\|x-x^*\|\leq r\}$ be the ball with center $x^*$ and radius $r>0$ and $r_0=\min\{r: \mathcal{L}_0\subset \mathbb{B}(x^*,r)\}$.
 \begin{theorem}\label{TNewt5}
 If (\ref{Eq_4.2.16})  and  (\ref{Eq2_4.2.16}) are  satisfied and $\delta=\frac{2}{3}\frac{m}{M+2L}$, then from any given starting point $x_0\in\mathcal{L}_0$ it takes 
 \begin{equation}\label{Eq2_19.2.16}
 N_0=13.5\left( \frac{L^2(M+2L)^3}{(m_0m)^3}(1+ r_0) (f(x_0)-f(x^*))\right)
 \end{equation}
 
DRN steps to get $x\in\mathbb{B}(x^*,\delta)$.
 \end{theorem}
 
 \begin{proof}
  For the regularized Newton's directional derivative  we have 
 $$ \frac{df(x+tr(x))}{dt}|_{t=0}=(\nabla f(x), r(x))=$$
$$ -\left((\nabla ^2 f(x)+\| \nabla f(x) \|I) r(x), r(x)\right) \leq $$
 \begin{equation}
  -(m(x)+\|\nabla f(x)\|) \|r(x) \|^2,\label{1.21}
 \end{equation}
 where $m(x) \geq 0$ and $\|\nabla f(x)\|>0$ for any $x\neq x^*$. It means that RND is a decent direction at any $x \in \mathcal{L}_0$ and $x\neq x^*$.

 It follows from (\ref{1.21}) that $\varphi (t) = f(x+tr(x))$ is monotone decreasing for small $t>0$.\\
  
 From the convexity of $f$ follows that  $\varphi '(t) = (\nabla f(x+tr(x)), r(x)) $ is not decreasing in $t>0$, hence at some  $t=t(x)$ we have
 \begin{equation}
 (\nabla f(x+t(x)r(x)), r(x))\geq-\frac{1}{2} (m(x)+\|\nabla f(x) \|) \|r(x) \| ^2, \label{1.24}
 \end{equation}
 otherwise $\inf f(x)= -\infty,$ which is impossible due to the boundedness of $\mathcal{L}_0$.
 
 From (\ref{1.21}) and (\ref{1.24}) we have
 $$(\nabla f(x+t(x)r(x))-\nabla f(x), r(x))\geq \frac{m(x)+\|\nabla f(x)\|}{2} \| r(x)\|^2.$$
 Therefore there exist $0<\theta (x) <1$ such that 
 $$ t(x) (\nabla^2 f(x+\theta (x) t(x) r(x)) , r(x) )=t(x)(\nabla^2f(\cdot)r(x),r(x))$$ $$\geq \frac{m(x)+\|\nabla f(x)\|}{2} \|r(x)\|^2$$
 or
 $$t(x) \| \nabla ^2 f(\cdot)\| \| r(x) \|^2 \geq \frac{m(x)+\|\nabla f(x)\| }{2} \|r(x) \|^2.$$
 Keeping in mind that $\|\nabla ^2 f(\cdot)\| \leq L$ we obtain 
 \begin{equation}
 t(x)\geq \frac{m(x) +\|\nabla f(x) \|}{2L}\geq \frac{\|\nabla f(x) \|}{2L}. \label{1.25}
 \end{equation} 
It means that for $t\leq\frac{\|\nabla f(x)\|}{2L}$ the inequality $$\frac{df(x+tr(x))}{dt}\leq -\frac{1}{2}(\nabla f(x),r(x))$$ holds, hence
 $$ \Delta f(x)=f(x) -f(x+t(x) r(x))\geq$$
 \begin{equation}
 \frac{1}{2} t(x) (-\nabla f(x), r(x) )= \frac{1}{2} t(x) \lambda _r^2 (x).\label {1.26}
 \end{equation}
 
 Therefore  finding the lower bound for  the decrease of $f$ at any $x\in \mathcal{L}_0$ such that $ x \notin B(x^*, \delta$) we have to find the corresponding    bound for the regularized Newton's decrement.

 Now let us   consider $x \in  B (x^*,\delta) $ then from \eqref{Eq5_12.3.16} follows
 \begin{equation}\label{Eq6_12.3.16}
 (\nabla f(x)-\nabla f(x^*), x-x^*)\geq m_0\|x-x^*\|^2.
 \end{equation}
 for any $x\in B(x^*,\delta)$.
 
 Let $\hat{x} \notin B(x^*, \delta)$, we consider a segment $[x^*, \hat{x}]$.
 There is $0< \widetilde{t}<1$ such that $ \widetilde{x}=(1-\widetilde{t})x^*+\widetilde{t}\hat{x} \in \partial B(x^*,\delta)$.
 
 From the convexity  $f$ follows 
 \begin{eqnarray*}
 (\nabla f(x^* +t(\hat{x}-x^*)),\hat{x}-x^*) |_{t=0} \leq (\nabla f (x^*+t(\hat{x}-x^*)), \hat{x}-x^*)|_{t=\widetilde{t}}\leq\\
  (\nabla f(x^*+t(\hat{x}-x^*),\hat{x}-x^*)|_{t=1},
 \end{eqnarray*}
 or
 $$0=(\nabla f(x^*), \hat{x}-x^*)\leq (\nabla f(\widetilde{x}), \hat{x}-x^*)\leq (\nabla f(\hat{x}), \hat{x}-x^*).$$
 The right inequality can be rewritten as follows:
 $$(\nabla f(\widetilde{x}), \hat{x}-x^*)= \frac{\|\hat{x}-x^*\|}{\delta}(\nabla f(\widetilde{x})-\nabla f(x^*), \widetilde{x}-x^*)\leq (\nabla f(\hat{x}), 
 \hat{x}-x^*).$$
 
 In view of (\ref{Eq6_12.3.16}) we obtain
 $$\|\nabla f(\hat{x})\|\|\hat{x}-x^*\| \geq \frac{\|\hat{x}-x^*\|}{\delta}(\nabla f(\widetilde{x})-f(x^*),\widetilde{x}-x^*)\geq \frac{\|\hat{x}-x^*\|}{\delta}m_0\|\widetilde{x}-x^*\|^2.$$
 Keeping in mind that $\widetilde{x} \in \partial B(x^*,\delta)$ we get 
 \begin{equation}
 \|\nabla f(\hat{x})\| \geq m_0\|\widetilde{x}-x^*\|=\frac{2}{3} m_0 m\frac{1}{M+2L}.\label{1.29}
 \end{equation}
 On the other hand from  \eqref{Eq4_5.2.16}  and $\hat{x}\in\mathcal{L}_0$ follows
\begin{equation}
\label{Eq_13.3.16}
\|\nabla f(\hat{x}) \|=\| \nabla f(\hat{x})-\nabla f(x^*)\|\leq L\|\hat{x}-x^*\|\leq Lr_0.
\end{equation}
From (\ref{Eq3_5.2.16}) follows 
\begin{equation}\label{Eq6_2.3.16}
\nabla  ^2 f(x)\preceq LI.
\end{equation}
For any $\hat{x}\notin S(x^*,\delta)$ we have $\|\nabla f(\hat{x})\|>0$, therefore  $H(\hat{x})=\nabla^2 f(\hat{x})+\|\nabla f(\hat{x})\|I$ is positive definite and system \eqref{1.6} has a unique solution 
$$r(\hat{x})=-H^{-1}(\hat{x})\nabla f(\hat{x}).$$
Moreover from (\ref{Eq6_2.3.16}) follows
$$(\nabla ^2 f(\hat{x})+\|\nabla f(\hat{x})\|I)\preceq (L+\|\nabla f(\hat{x})\|)I.$$
Therefore 
\begin{equation}
H^{-1}(\hat{x})\succeq (L+\|\nabla f(\hat{x})\|I)^{-1})I.
\end{equation}
For the regularized Newton's decrement we obtain 
\begin{equation}\label{Eq7_2.3.16}
\lambda_{(r)}(\hat{x})=(H^{-1}(x))\nabla f(\hat{x}), \nabla f(\hat{x}))^{0.5}\geq (L+\|\nabla f(\hat{x}\|)^{-0.5}\|\nabla f(\hat{x})\|.
\end{equation}
Keeping in mind 
$$\|\nabla f(\hat{x}\|=\|\nabla f(\hat{x})-\nabla f(x^*)\|\leq L\|\hat{x}-x^*\|$$
from (\ref{1.25}), \eqref{Eq_13.3.16} and (\ref{Eq7_2.3.16}) and definition of $r_0$ we obtain 
$$\Delta f(\hat{x})\geq \frac{1}{2} t(\hat{x})\lambda^2_r(\hat{x})\geq \frac{\|\nabla  f(\hat{x})\|^3}{4L}(L+\|\nabla f(\hat{x})\|)^{-1}\geq \frac{\|\nabla f(\hat{x})\|^3}{4L^2(1+r_0)}.$$
Using \eqref{1.29} we  get 
$$\Delta f(\hat{x})\geq \left( \frac{2}{3} m_0m\frac{1}{M+2L}\right)^3\frac{1}{4L^2(1+r_0)}$$ $$=\frac{2}{27}\frac{(m_0m)^3}{(M+2L)^3L^2}\frac{1}{(1+r_0)}.$$
Therefore  it takes 
$$ N_0=(f(\hat{x})-f(x^*))\Delta f^{-1}(\hat{x})=13.5\frac{(M+2L)^3L^2}{(m_0m)^3}(1+r_0)(f(\hat{x})-f(x^*))$$
steps to obtain $x\in \mathbb{B}(x^*,\delta)$ from any $x\in \mathcal{L}_0$.

 The proof of  Theorem  \ref{TNewt5} is completed.\qed
 \end{proof}
 From \eqref{1.14a} follows that it takes $O(\ln\ln \varepsilon^{-1})$ DRN steps to find an $\varepsilon$-approximation for $x^*$ from any $x\in \mathbb{B}(x^*,\delta)$. 
 
 Therefore the total number of DRN steps required for finding an $\varepsilon$-approximation for $x^*$ from a given starting point $x_0\in \Rn$ is 
 $$N=N_0+o(\ln\ln\varepsilon^{-1}).$$

 
 \section{Concluding Remarks}
 The bounds \eqref{Eq7_18.9.15} and \eqref{Eq2_19.2.16} depends on the size of Newton's and regularized Newton's areas, which, in turn, are defined by  convexity constant $m>0$ and smoothness constants $M>0$ and $L>0$. The convexity and smoothness constants dependent on the given  system of coordinate.
 
 Let consider an affine transformation of the original system  given by $x=Ay$, where $A\in \RR^{n\times n}$ is a nondegenerate matrix. We obtain  $\varphi (y)=f(Ay).$
 
 Let $\{x_s\}^{\infty}_{s=0}$ be the sequence generated by Newton's method
 $$x_{s+1}=x_s-(\nabla^2 f(x_s))^{-1}\nabla f(x_s).$$
 
 For the correspondent sequence in the transformed space we obtain
 $$y_{s+1}=y_s-(\nabla^2\varphi(y_s))^{-1}\nabla \varphi(y_s).$$
 Let $y_s=A^{-1}x_s$ for some $s\geq 0$, then $$y_{s+1}=y_s-(\nabla^2\varphi(y_s))^{-1}\nabla \varphi (y_s)=y_s-[A^T\nabla^2f(Ay_s)A]^{-1}A^T\nabla f(Ay_s)=$$
 $$A^{-1}x_s-A^{-1}(\nabla^2 f(x_s))^{-1}\nabla f(x_s)=A^{-1}x_{s+1}.$$
 It means that Newton's method is affine invariant with respect to the transformation $x=Ay$. Therefore the areas of quadratic convergence depends only on the local topology of $f$(see \cite{7}).
 
To get the Newton's sequence in the transformed space one needs to apply $A^{-1}$ to the elements of the Newton's  sequence in the original space.
 
 Let $N$ is such that $x_N:\|x_N-x^*\|\leq \varepsilon$, then   
 $$\|y_N-y^*\|\leq \|A^{-1}\| \| x_N-x^*\|.$$
 From \eqref{Newt15} follows
 $$\|x_{N+1}-x^*\|\leq \frac{M}{2(m-M\|x_s-x^*\|)}\|x_N-x^*\|^2.$$
 Therefore  
 $$\|y_{N+1}-y^*\|\leq \|A^{-1}\| \|x_{N+1}-x^*\|\leq \frac{1}{2}\|A^{-1}\|\frac{M}{(m-M\varepsilon)}\varepsilon^2.$$
 Hence, for small enough 
 $$\varepsilon \leq 0.5\frac{m}{M}\min \{1;(\|A^{-1}\|)^{-1}\}$$ we have 
 $$\|y_{N+1}-y^*\|\leq \varepsilon.$$

 We would like to emphasize that the bound \eqref{Eq2_19.2.16} is global, while the conditions \eqref{Eq_4.2.16} and \eqref{Eq2_4.2.16}  under which the bound holds are local, at the neighborhood of $x^*$.

\end{document}